\documentclass[11pt,leqno]{article}
\usepackage[a4paper, total={6in, 8in}]{geometry}
\usepackage[T1]{fontenc}
\usepackage{amssymb}
\usepackage{amsmath}
\usepackage[title, toc]{appendix}
\usepackage{amsthm}  
\usepackage{tikz-cd}  
\usepackage{tikz,stackengine}  
\usepackage{enumerate} 
\usepackage{enumitem}  
\usepackage{pdfpages}  
\usepackage{hyperref}
\usepackage{wasysym}  
\usepackage{mathtools} 
\usepackage{stmaryrd} 
\usepackage{graphicx}
\usepackage{xcolor}
\usepackage{mathrsfs}
\usepackage{titlesec}
\usepackage{ytableau}
\usepackage{ bbold }
\usepackage{mathtools}
\usepackage{dynkin-diagrams}
\usepackage[matrix,arrow,curve]{xy}
\usepackage{xcolor}

\hypersetup{
    colorlinks,
    linkcolor={blue!50!black},
    citecolor={blue!50!black},
    urlcolor={blue!80!black}
}
\makeatletter
\def\th@remark{%
  \thm@headfont{\bfseries}%
  \normalfont 
  \thm@preskip\topsep \divide\thm@preskip\tw@
  \thm@postskip\thm@preskip
}

\makeatletter \newcommand{\myitem}[2][]{%
\item[#2#1]\protected@edef\@currentlabel{#2}%
}
 \makeatother

\makeatother
\DeclareSymbolFont{extraitalic} {U}{zavm}{m}{it}
\DeclareMathSymbol{\stigma}{\mathord}{extraitalic}{168}
\newtheorem*{proposition*}{Proposition}
\newtheorem*{theorem*}{Theorem}
\newtheorem{theorem}{Theorem}[subsection]
\newtheorem{proposition}[theorem]{Proposition}

\newtheorem{lemma}[theorem]{Lemma}
\newtheorem{corollary}[theorem]{Corollary}
\theoremstyle{definition}
\newtheorem{definition}[theorem]{Definition}
\theoremstyle{definition}

\theoremstyle{remark}
\newtheorem{remark}[theorem]{Remark}

\newcommand*\leftdash{\rotatebox[origin=c]{-45}{$\dabar@\dabar@\dabar@$}}
\newcommand*\rightdash{\rotatebox[origin=c]{45}{$\dabar@\dabar@\dabar@$}}
\newcommand*\mondash[1]{\rotatebox[origin=c]{45}{$\dabar@\dabar@\dabar@$}{\kern-1.5ex\raisebox{-.5ex}{}_{#1}}\kern.5ex}
\makeatother

\newcommand{\He}{\mathcal{H}}
\newcommand{\Se}{\mathbb{S}}

\newcommand{\aff}{\mathrm{aff}}
\newcommand{\tr}{\mathrm{tr}}

\titleformat{\section}
{\normalfont\bfseries}{\thesection}{1em}{}

\titleformat{\subsection}[runin]
{\normalfont\bfseries}{\thesubsection}{1em}{}

\titleformat{\subsubsection}[runin]
{\normalfont\bfseries}{\thesubsection}{1em}{}

\newcommand{\Addresses}{{
  \bigskip
  \footnotesize
  K.~Tolmachov, \textsc{Department of Mathematics, University of Hamburg,
    Hamburg, Germany}\par\nopagebreak
  \textit{E-mail address}: \texttt{tolmak@khtos.com}

  \medskip

   H.~Zhylinskyi, \textsc{Department of Mathematics, Jagiellonian University, Krakow, Poland}\par\nopagebreak
  \textit{E-mail address}: \texttt{georgyzhilinsky144@gmail.com}
}}

\begin{document} \title{Generalized Markov traces and Jucys-Murphy elements}

\author{Kostiantyn Tolmachov, Heorhii Zhylinskyi} \date{}
\maketitle \abstract{We give a simple construction of Markov traces for Iwahori-Hecke algebras associated with infinite series of crystallographic Coxeter groups. In types $B$ and $D$ it is new, and generalizes a known construction in type $A$ employing symmetric polynomials in multiplicative Jucys-Murphy elements.}

\section{Introduction.}
A two-variable link invariant -- the HOMFLY-PT polynomial -- can be recovered using a family of trace functions on Iwahori-Hecke algebras in type $A$ satisfying analogues of the conditions in Markov's theorem for the braid group (usually called ``Markov moves'') \cite{jonesHeckeAlgebraRepresentations1987}. The analogues of Markov moves can be formulated for Iwahori-Hecke algebras of other Coxeter groups, leading to a general notion of the Markov trace, see, for example, \cite{geckTraceFunctionsIwahoriHecke1998}, \cite{geckMarkovTracesKnot2004}, \cite{gomi2006markov},  \cite{websterGeometryMarkovTraces2009}. The Markov traces where classified in type $A$, where there is, up to some change of scaling, a unique Markov trace, by Ocneanu and Jones \cite{jonesHeckeAlgebraRepresentations1987}, and in types $B$ and $D$ by Geck and Lambropoulou \cite{geckMarkovTracesKnot2004}, \cite{geckTraceFunctionsIwahoriHecke1998}. In type $B$, Markov traces produce the invariants of the solid torus links \cite{lambropoulou1994solid} \cite{geckMarkovTracesKnot2004}.

Of particular interest are specializations of the general Markov traces in types $B, D$ coming from geometry (or, alternatively, from the theory of Soergel bimodules). These where studied in \cite{gomi2006markov} using the Lusztig's Fourier transform, in \cite{websterGeometryMarkovTraces2009} using the equivariant Hecke category, and in \cite{bezrukavnikovMonodromicModelKhovanov2022} using the monodromic Hecke category. Note that the latter two papers work in the setting of an arbitrary Weyl group and allow to categorify Markov traces to obtain a triply-graded link invariant known as Khovanov-Rozansky homology. Its earlier construction, using Soergel bimodules in type $A$ \cite{khovanovTriplygradedLinkHomology2007}, readily generalizes to an arbitrary Coxeter group.

Iwahori-Hecke algebras in types $A, B, D$ contain families of commuting operators whose properties are analogous to those of Jucys-Murphy elements in the group algebra of the symmetric group. 
Under a standard normalization, the Markov trace in type $A$ is polynomial in variable $a$. Its individual coefficients turn out to be ``represented'' (in the sense explained in the main text) by the elementary symmetric polynomials in multiplicative Jucys-Murphy elements. It was shown in \cite{bezrukavnikovMonodromicModelKhovanov2022} that this fact admits a categorification to the level of Khovanov-Rozansky homology. Moreover, in type $A$, multiplicative Jucys-Murphy elements themselves admit an interesting categorification, defining relative Serre functors with respect to the embedding of a smaller Hecke category into a larger one \cite{gorskySerreDualityKhovanov2019}, \cite{ho2025relative}.  Our investigation was motivated by a desire to see how far these patterns can be extended for other types.

In this paper, we give a new simple construction for Markov traces in types $B, D$ using the theory of Jucys-Murphy elements in Iwahori-Hecke algebras of these types.

The paper is organized as follows.
In Section \ref{sec:iwah-hecke-algebr-2} we collect the needed definitions about Coxeter groups and Iwahori-Hecke algebras. In Section \ref{sec:some-centr-elem} we define a family of commuting elements in Iwahori-Hecke algebras of type $A, B, D$ coming from the extended affine Hecke algebra in type $A$. See \cite{ram1997seminormal} for the history and an alternative definition. We follow \cite{Ram2003} in our treatment of these elements.

In Section \ref{sec:markov-traces-1} we first define the Markov traces and give, as a warm-up, a construction of a Markov trace defined by a uniform formula in all types $A, B, D$. See Theorem \ref{sec:markov-traces-via}. In Section \ref{sec:gener-mark-trac-3} we recall the results of Geck and Lambropoulou and give the formulas for the Markov traces in types $B$ and $D$ depending on a free parameter (Theorem \ref{sec:gener-mark-trac-2} and Corollary \ref{sec:gener-mark-trac-4}, respectively). In Section \ref{sec:geom-mark-trac} we specialize our results to obtain the geometric Markov traces.

\subsection{Acknowledgments.} The results of this paper were obtained while the authors participated in \href{https://math.mit.edu/research/highschool/primes/YuliasDream/}{Yulia's Dream} program. We would like to thank the program's or\-ga\-ni\-zers for this wonderful initiative. We would like to thank Mykhailo Barkulov, another participant of the program, for helpful discussions.

We have conducted numerical experiments using SageMath software \cite{sagemath} throughout the work on this paper, and would like to thank its developers.

K.T. was funded by the Deutsche Forschungsgemeinschaft SFB 1624 grant, Pro\-jekt\-nu\-mmer 506632645.

\section{Iwahori-Hecke algebras and multiplicative Jucys-Murphy elements.}
\label{sec:iwah-hecke-algebr-2}
\subsection{Coxeter groups.}
Let $(W, S), S = \{s_1, \dots, s_n\},$ be a Coxeter system. Thus, $W$ is a group with a presentation $(s_0, \dots, s_n | (s_is_j)^{m_{ij}} = e)$, where $m_{ij} \in \mathbb{Z}^+ \cup \{\infty\}$, with $m_{ii} = 1, m_{ji} = m_{ij} \geq 2$ if $i \neq j$, and $m_{ij} = \infty$ means that there is no relation included for the pair $s_i, s_j$. Such a presentation is read off the Coxeter diagram of $(W,S)$ in a standard way: such a diagram is a non-oriented graph with $S$ as a set of vertices, and $m_{ij} - 2$ edges between the vertices $s_i$ and $s_j$. Let $l: W \to \mathbb{Z}_{\geq 0}$ be the Bruhat length function with respect to $S$.

We will be primarily interested in Coxeter groups that are Weyl groups of classical types $A, B = C, D$. To fix the numeration of the generators, their Coxeter diagrams are presented below, with a vertex corresponding to $s_i \in S$ marked by $i$. The Coxeter diagram of type $A_n$ is \[\dynkin[arrows=false, edge length=0.6cm, backwards, labels={n,n-1,2,1}]A{},\] of type $B_n = C_n$ is
\[
  \dynkin[arrows=false, edge length=0.6cm, backwards, labels={n-1,,,1,0}]B{},
\]
and of type $D_n$ is 
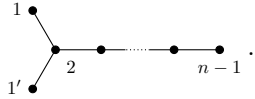
\[
\dynkin[arrows=false, edge length=0.6cm, backwards, labels={n-1,,,2,1,1'}]D{}.
\]
When $n$ is small, these are interpreted in a standard way as follows: the Coxeter system $A_1 = B_1 = D_1$ has a single vertex, $D_2 = A_1 \times A_1$, $D_3 = A_3$. We also adopt the convention $A_0 = B_0 = D_0$ corresponding to the empty Coxeter graph and the trivial Coxeter group.

Many of the results and definitions in this paper are valid in types all types $A, B, D$, with definitions for type $A_{n-1}$ naturally fitting with definitions in types $B_n, D_n$. In this case, we will write ``type $X_n$, where $X$ is $A_{-1}$, $B$ or $D$'' (or something in a similar vein).  

\subsection{Iwahori-Hecke algebras.}
Let $\mathcal{A} = \mathbb{Q}(v)$. The Iwahori-Hecke algebra $\He_v(W,S) = \He(W,S)$ associated to a Coxeter system $(W,S)$ is a unital $\mathcal{A}$-algebra with generators $t_s, s \in S,$ subject to the relations
\begin{enumerate}[label=(\roman*).,ref=(\roman*)]
\item $\langle t_{s_i},t_{s_j}\rangle_{m_{ij}} = \langle t_{s_j},t_{s_i}\rangle_{m_{ij}}$, where $\langle x,y \rangle_{m}$ is an alternating product $xyx\dots$ of $x$ and $y$ with $m$ factors.
\item \label{item:1} $(t_s - v)(t_s + v^{-1}) = 0.$
\end{enumerate}
When $(W,S)$ is of type $B_n$ we will also be interested in the Iwahori-Hecke algebra $\He_{v,v_0}(W,S)$ with unequal parameters $v, v_0 \in \mathcal{A}$. This is a unital algebra over the ring $\mathcal{A}$, with the relation \ref{item:1} replaced by the relations
\begin{enumerate}[]
\item[(ii\textquotesingle).] $(t_{s_i} - v)(t_{s_i} + v^{-1}) = 0$ for $i > 1$,
\item[(ii\textquotesingle\textquotesingle).] $(t_{s_i} - v_0)(t_{s_i} + v_{0}^{-1}) = 0$.
\end{enumerate}

When $(W,S)$ is of type $X_n$ with $X \in \{A, B=C, D\}$, we will write $W(X_n)$ and $\He(X_n)$ for $W$ and $\He(W,S)$, or $\He_{v,v_0}(B_n)$ if we need to emphasize that we use unequal parameters. In each of these cases, we will write $t_i$ for $t_{s_i}$.

Let $w = s_{i_1}\dots s_{i_r}, w \in W, s_{i_k} \in S$ be a reduced expression. Then the element $t_w = t_{s_{i_1}}\dots t_{s_{i_r}}$ in $\He(W,S)$ is independent of the choice of a reduced expression, and $\He(W,S)$ is a free $\mathcal{A}$-module with a basis $t_w, w \in W$. The elements $t_w$ are invertible, and the collection $t_{w}^{-1}, w \in W,$ gives another basis of $\He(W,S)$ over $\mathcal{A}$. 

\subsection{Trace functional and a pairing on $\He(W,S)$.} Let $x \mapsto \overline{x}$ be the Kazhdan-Lusztig involution on $\He(W,S)$. It is a $\mathbb{Q}$-algebra involution, defined by the assignments $\overline{v} = v^{-1}, \overline{t_w} = t_{w^{-1}}^{-1}.$ Let $i:\He(W,S)$ stand for the $\mathcal{A}$-linear algebra anti-involution defined by $i(t_w) = t_{w^{-1}}$. 

For $h \in \He(W,S), x \in W,$ let $h_x$ stand for the $t_x^{-1}$-coefficient of $h$ in the basis $t_w^{-1},$ where $w$ runs through $W$. Let $\tau: \He(W,S) \to \mathcal{A}$ be a functional defined by  $\tau(h) = h_e,$ where $e \in W$ is the unit.  For $h, h' \in \He(W,S)$, define a pairing $\langle h, h' \rangle = \tau(i(\overline{h})h')$. We have the following standard proposition, see, for example, \cite[Chapter 8]{geck2000characters}.
\begin{proposition}[]
  \label{sec:trace-funct-pair}
  The functional $\tau$ and the pairing $\langle\, , \rangle$ satisfies the following properties.
  \begin{enumerate}[label=(\alph*)., ref=(\alph*)]
  \item The functional $\tau$ is an $\mathcal{A}$-linear trace on $\He(W,S),$ that is $\tau(hh') = \tau(h'h).$
  \item The pairing $\langle\, , \rangle$ is $\mathcal{A}$-antilinear in the first argument, $\mathcal{A}$-linear in the second argument and satisfies the orthogonality relation
    \(   \langle t_{w_1} , t^{-1}_{w_2} \rangle = \delta_{w_1, w_2^{-1}}.
    \)
   \item \label{item:3} We have \(\langle h_1h_2, h_3 \rangle = \langle h_1, h_3 i(\overline{h_2})\rangle = \langle h_2,  i(\overline{h_1})h_3\rangle\) for any $h_1, h_2, h_3 \in \He(W,S).$
  \end{enumerate}
\end{proposition}
We will use the following straightforward corollary of this proposition repeatedly.
\begin{corollary}
  \label{sec:trace-funct-pair-1}
  Let $(W',S')$ be a parabolic subgroup of $W$. Assume that $x, y \in \He(W',S') \subset \He(W,S)$, and assume that $z \in \mathrm{span}_\mathcal{A}\langle t_w, w \in W\backslash W'\rangle$ (resp. $z \in \mathrm{span}_\mathcal{A}\langle t_w^{-1}, w \in W\backslash W'\rangle$). Then $\langle xz, y \rangle = \langle zx, y \rangle = 0$ (resp. $\langle x, yz \rangle = \langle x, zy \rangle =0$).
\end{corollary}

\subsection{``Full twist'' and Jucys-Murphy elements.}\label{sec:full-twist-jucys} Assume that $(W, S)$ is of finite type. Let $w_0$ stand for the longest element in $W$ with respect to the Bruhat length. Let $\Se(W,S) = t_{w_0}^{-2} \in \He(W,S)$. In classical types, we will write $\Se(X_n) \in \He(X_n)$, when $X$ is $A, B$ or $D$, or simply $\Se$ when the type is clear from the context.  

The following proposition is also standard:
\begin{proposition}
  \label{sec:iwah-hecke-algebr}
  The element $\Se(W,S)$ is central in $\He(W,S)$. 
\end{proposition}

We will need the following ``Serre property'' of the operator of multiplication by $\Se(W,S)$.
\begin{proposition}
  \label{sec:full-twist-jucys-1}
  For any $h_1, h_2 \in \He(W,S)$ we have the identity
  \[
   \langle h_1, \Se(W,S)h_2 \rangle = \overline{\langle h_2, h_1 \rangle}.
  \]
\end{proposition}
\begin{proof}
  It is enough to prove the proposition for $h_1 = t_{w_1}^{-1}, h_2 = t_{w_2}$ for arbitrary $w_1, w_2 \in W$. We have, by Proposition \ref{sec:trace-funct-pair}~\ref{item:3},
  \begin{multline*}
    \langle t_{w_1}^{-1}, \Se(W,S)t_{w_2} \rangle = \langle t_{w_0}t_{w_1}^{-1}, t_{w_0}^{-1}t_{w_2}\rangle = \langle t_{w_0w_1^{-1}}, t_{w_2^{-1}w_0}^{-1} \rangle =\\
    = \delta_{w_0w_1^{-1},w_0w_2} = \delta_{w_1,w_2^{-1}} = \overline{\langle t_{w_2}, t_{w_1}^{-1}\rangle}.
\end{multline*}
This completes the proof.
\end{proof}

We have the natural embedding $\iota_n:\He(X_n) \to \He(X_{n+1})$, where $X$ is $A, B$ or $D$, induced by the embedding of the sets of labels of the Coxeter diagram. Using these embeddings, an element of $\He(X_n)$ can be considered as an element of $\He(X_m)$ for $m > n$. 

\begin{definition}
  For $n \geq 2$, let $J_n^X = \Se(X_{n})^{-1}\cdot \Se(X_{n-1})$ when $X$ is $A_{-1}, B$ or $D$. We let $J_1^A = 1, J_1^B = t_0^2$, $J_1^D = 1$. We call these elements \emph{Jucys-Murphy elements} of finite Hecke algebras in classical types.

\end{definition}

We have the following corollary of Proposition \ref{sec:iwah-hecke-algebr}.
\begin{corollary}
  \label{sec:iwah-hecke-algebr-1}
   For $X$ equal to $A_{-1}, B$ or $D$, elements $J_k^X$ commute with each other when considered lying in the same Iwahori-Hecke algebra via the embedding described above. The element $J_n^X$ lies in the centralizer of $\He(X_{n-1})$ in $\He(X_{n})$.
\end{corollary}
\begin{remark}
  The elements $J_k^A$ are sometimes called multiplicative Jucys-Murphy elements in the Iwahori-Hecke algebra of type $A$. There seem to be no agreement in the literature on which elements to call the Jucys-Murphy elements in other types. We chose our convention because under it the definition of the Jucys-Murphy element is uniform in all types, but elements $j_k^X$ for $X = B, D$ defined below arguably also deserve the name.  
\end{remark}

\subsection{Some central elements in the classical Iwahori-Hecke algebras.}
\label{sec:some-centr-elem}
In this section we follow \cite{Ram2003}.
Let $\He^{\aff}(A_{n-1})$ be the affine Hecke algebra associated to the root datum of $\mathrm{GL}_n$. It is a unital algebra over $\mathcal{A}$ generated by the elements $t_1, \dots, t_{n-1}$ and $\Theta_1^{\pm 1}$, where $t_1, \dots, t_{n-1}$ satisfy the relations of the Iwahori-Hecke algebra of type $A_{n-1}$ and, additionally, the following relations hold:
\begin{enumerate}
\item[(iii).] $t_1\Theta_1 t_1\Theta_1 = \Theta_1 t_1 \Theta_1 t_1$,
\item[(iv).] $t_i\Theta_1 = \Theta_1 t_i$ for $i > 1$.
\end{enumerate}
Inductively define $\Theta_i = t_i\Theta_{i+1}t_i \in \He^{\aff}(A_{n-1}), i = 2, \dots, n$. It is well-known that $\Theta_i^{\pm 1}$ generate a commutative subalgebra of $\He^{\aff}(A_{n-1})$, called the Bernstein subalgebra. We have the following particular case of a classical theorem, proved in \cite[Theorem 8.1]{lusztigSingularitiesCharacterFormulas1983} for all affine types.  
\begin{theorem}
 The center of $\He^{\aff}(A_{n-1})$ consists of the symmetric Laurent polynomials in $\Theta_1, \dots, \Theta_{n}$. 
\end{theorem}

Fix $n \geq 2$. Note that
\begin{enumerate}[label=(\roman*).,ref=(\roman*)]
\item The Hecke algebra $\He(A_{n-1})$ is the quotient of the affine Hecke algebra $\He^{\aff}(A_{n-1})$ by the relation $\Theta_1 = 1$.
\item\label{item:11} The Hecke algebra $\He_{v,v_0}(B_n)$ is the quotient of $\He^{\aff}(A_{n-1})$ by the relation \[(\Theta_1 - v_0)(\Theta_1 + v_0^{-1}) = 0.\] Here $t_0$ is the image of $\Theta_1$.
\item The Hecke algebra $\He(D_n)$ is the subalgebra of the quotient \ref{item:11} with $v_0 = 1$,  ge\-ne\-ra\-ted by the images of $\Theta_1 t_1 \Theta_1^{-1}, t_1, \dots, t_{n-1}$. Here $t_{1'}$ is the image of $\Theta_1 t_1 \Theta_1^{-1}$.
\end{enumerate}
Note that this exhibits $\He(D_n)$ as a subalgebra of $\He_{v,1}(B_n)$.
Note also that $J^A_i$ is the image of $\Theta_i$ in $\He(A_{n-1})$ for $i = 1, \dots, n$. We have \[J^A_{i} = t_{i-1}\dots t_1t_1 \dots t_{i-1}.\]

\begin{definition}
  We define the following families of elements in the Hecke algebras of finite types $B, D$:
  \begin{enumerate}[]
  \item[(B).] Let $j^B_i$ be the image of $\Theta_i, i = 1, \dots, n,$ in $\He(B_{n})$. We have \[j^B_{i} = t_{i-1}\dots t_1t_0t_1 \dots t_{i-1}.\]
    and $J^B_i = (j^B_i)^2.$
  \item[(D).] Let $j^D_i$ be the image of $\Theta_i\Theta_1^{-1}, i = 1, \dots, n,$ in $\He(D_{n})$. We have \[j^D_{i} = t_{i-1}\dots t_2t_1't_1t_2 \dots t_{i-1}, i \geq 2, j_1^D = 1.\]
    and $J^D_i = (j^D_i)^2.$
  \end{enumerate}
\end{definition}
We have the following
\begin{corollary}
 \label{sec:jucys-murphy-central} 
 The elements $J^A_i$ (resp. $j^B_i, j^D_i$) commute with each other. If $P$ is a symmetric polynomial in $n$ variables, $P(J^A_1, \dots, J^A_n)$ (resp. $P(j^B_1, \dots, j^B_n)$) is central in $\He(A_n)$ (resp. in $\He_{v,v_0}(B_n)$). If $P$ is homogeneous of even degree, then $P(j^D_1, \dots, j^D_n)$ is central in $\He(D_n)$. 
\end{corollary}
\section{Markov traces.}
\label{sec:markov-traces-1}
\subsection{Markov traces via Jucys-Murphy elements.}
\label{sec:markov-traces-via-1}
\begin{definition}
\label{sec:markov-traces}  
Let $\mathcal{R}$ be a commutative $\mathcal{A}$-algebra. A family $(\phi_n)$ of $\mathcal{A}$-linear functionals $\phi_n: \He(X_n),$ where $X$ is $A_{-1}, B$ or $D$ and $n \in \mathbb{Z}_{\geq 1}$ ($n \in \mathbb{Z}_{\geq 2}$ in type $D$) is called a Markov trace if there are elements $\mu, \rho \in \mathcal{R}$ such that the following conditions are satisfied.
\begin{enumerate}
\myitem[.]{(M1)}\label{item:4} $\phi_n$ is a trace function, i.e. $\phi_n(hh') = \phi_n(h'h)$. This property is a generalization of the first Markov move from the theory of link invariants. 

  Recall the embedding $ \iota = \iota_{n-1}:\He(X_{n-1}) \to \He(X_n)$ defined in Section \ref{sec:full-twist-jucys}. The functionals $\phi_{n-1}$ and $\phi_n$, $n \geq 2$ must satisfy 
\myitem[.]{(U1)}\label{item:5} 
  \(   \phi_n(\iota(h)) = \rho \phi_{n-1}(h). \) This property is a generalization of a rule for removing an unlinked circle in the theory of link invariants.
\myitem[.]{(M2)}\label{item:6} 
  \(   \phi_n(\iota(h)t_n) = \mu \phi_{n-1}(h). \)
This property is a generalization of the second Markov move from the theory of link invariants.
\end{enumerate}

Note that if $(\phi_n)$ is a Markov trace with constants $\rho, \mu$,  then, for arbitrary $\lambda_1, \lambda_2 \in \mathcal{R}$, $(\lambda_1\lambda_2^n\phi_n)$ is again a Markov trace with constants $\lambda_2\rho, \lambda_2\mu$.

\begin{remark}
  We exclude small $n$ in our definition of the Markov trace because some special treatment is needed for the embeddings $B_0 \to B_1, D_0 \to D_1 \to D_2$. For the geometric traces these properties are extended to small $n$, see Section \ref{sec:geom-mark-trac}.
\end{remark}

\end{definition}
\begin{definition}
  Assume that the involution $v \mapsto v^{-1}$ is extended to $\mathcal{R}$. We say that an element $z \in \He(W,S) \otimes_\mathcal{A}\mathcal{R}$ represents a functional $\phi:\He(W,S) \to \mathcal{R}$ if $\phi(h) = \langle z, h\rangle$ for any $h \in \He(W,S)$. It is easy to see that $\phi$ satisfies the property \ref{item:4} if and only if $z$ is in the center of $\He(W,S)$. In this case we write $\phi = \tr_z$.
\end{definition}

Let $\mathcal{R} = \mathcal{A}(a)$ and let $\overline{a} = a^{-1}$. Define
\[
  \zeta_n = \prod_{i = 1}^{n} (1 + a^{-1}J^X_i), \zeta_0 = 1,
\]
an element in $\He(X_n)$ when $X = A_{-1}, B$ or $D$. By Corollary \ref{sec:jucys-murphy-central}, elements $\zeta_n$ are central in the corresponding Iwahori-Hecke algebras.  We prove 
\begin{theorem}
  \label{sec:markov-traces-via}
  The family of functionals $(\tr_{\zeta_n})$ is a Markov trace with constants $\rho = 1 + a, \mu = v-v^{-1}.$
\end{theorem}
\begin{proof}
  We first check the property \ref{item:5}. We have
  \[
    \langle \zeta_n, \iota(h)\rangle = \langle\zeta_{n-1}(1+a^{-1}J^X_n), \iota(h) \rangle = \langle \zeta_{n-1}, \iota(h)\rangle + a\langle\zeta_{n-1}J_n^X, \iota(h)\rangle. 
  \]
  Applying Proposition~\ref{sec:full-twist-jucys-1} twice --- recall that $J_n^X = \Se(X_n)^{-1}\Se(X_{n-1})$ --- we get the equality \(\langle\zeta_{n-1}J_n^X, \iota(h)\rangle = \langle\zeta_{n-1}, \iota(h)\rangle\) and \ref{item:5} follows.

  We now check the property \ref{item:6}. We have
   \begin{multline}
\label{eq:1}
\langle \zeta_n, \iota(h)t_n\rangle = \langle\zeta_{n-1}(1+a^{-1}J^X_n), \iota(h)t_n \rangle = \langle \zeta_{n-1}, \iota(h)t_n\rangle + a\langle\zeta_{n-1}J_n^X, \iota(h)t_n\rangle = \\
= \langle \zeta_{n-1}, \iota(h)t_n^{-1}\rangle +(v-v^{-1}) \langle \zeta_{n-1}, \iota(h)\rangle + a\langle\zeta_{n-1}J_n^X, \iota(h)t_n\rangle. 
\end{multline}
From Corollary \ref{sec:trace-funct-pair-1}, we get that
\(\langle\zeta_{n-1},\iota(h)t_n^{-1}\rangle = 0.\)

  For the third summand in the last expression in \eqref{eq:1}, applying Proposition \ref{sec:full-twist-jucys-1}, we have
  \[
    \langle\zeta_{n-1}J_n^X, \iota(h)t_n\rangle = \overline{\langle \iota(h)t_n,\zeta_{n-1}\Se(X_{n-1})\rangle} = 0,\]
with the last equality again following from Corollary \ref{sec:trace-funct-pair-1}. We obtain \ref{item:6}.
\end{proof}
\begin{remark}
 We will show below that, in types $A$ and $B$ (but not in type $D$), the traces $\tr_{\zeta_n}$ coincide with the traces coming from geometry. See Section \ref{sec:geom-mark-trac}.
\end{remark}
\subsection{General Markov traces in types $B, D$.}
\label{sec:gener-mark-trac-3}
In type $A$, the classical result of Jones and Ocneanu \cite{jonesHeckeAlgebraRepresentations1987} says that a Markov trace with given constants $\rho, \mu$ is defined uniquely up to an overall scaling. Thus, Theorem~\ref{sec:markov-traces-via} gives a general description of a Markov trace in type $A$.

In types $B, D$, the general description of Markov traces was obtained by Geck and Lambropoulou \cite{geckMarkovTracesKnot2004}, \cite{geckTraceFunctionsIwahoriHecke1998}. We recall their results.

For $n \geq 1$, define the elements
\[T_n = t_{n-1}\dots t_1 t_0 t_1^{-1}\dots t_{n-1}^{-1}\in \He(B_n), n \geq 2, T_1 = t_0,\]
and
\[U_n = t_{n-1}\dots t'_1 t_1^{-1}\dots t_{n-1}^{-1}\in \He(D_n), n \geq 2, U_1 = 1.\]

\begin{theorem}[\cite{geckMarkovTracesKnot2004}, \cite{geckTraceFunctionsIwahoriHecke1998}]
  \label{sec:gener-mark-trac-1}
 Let $\mathcal{R} = \mathcal{A}(a,y_1,\dots,y_n,\dots)$, with infinite number of variable $y_k$. There is a unique Markov trace $(\tr^X_{n,(y_{\bullet})})$, where $X$ is $B$ or $D$, with values in $\mathcal{R}$, with constants $\rho = 1 + a, \mu = v-v^{-1}$, normalization $\tr^{X}_{n,(y_\bullet)}(1) = (1+a)^n$ and additional properties:
  \begin{enumerate}
  \myitem[.]{(B)}\label{item:2} $\tr^{B}_{n,(y_{\bullet})}(T_1T_2 \dots T_n) = y_n$, $n \geq 1$, in type $B$.
  \myitem[.]{(D)}\label{item:7} $\tr^{D}_{2n,(y_{\bullet})}(U_1U_2 \dots U_{2n}) = y_{2n}$, $n \geq 1$, in type $D$.
  \end{enumerate}
Any Markov trace with the chosen $\rho, \mu$ and normalization is of this form.  Moreover, every such Markov trace in type $D$ is obtained by restriction from the Markov trace on $\He_{v,1}(B_n)$ (see Section \ref{sec:some-centr-elem}).
\end{theorem}

We will be interested in a particular case of these traces when $y_n$ is of the form $y_n = y^n.$ Let from now on $\mathcal{R} = \mathcal{A}(y)$ and denote by $\tr^{X}_{n,y}$ the Markov trace from Theorem \ref{sec:gener-mark-trac-1} with $y_n = y^n$ for all $n \geq 1$. Extend the bar involution to $\mathcal{R}$, formally writing $\overline{y}$ --- we assume that such an extension exists under any specialization of $y$ we will use. To simplify the formulas, we adopt the notation $\alpha = v - v^{-1}, \alpha_0 = v_0 - v_0^{-1}$.
Define the elements
\[
  \beta_n(y) = \prod_{i = 1}^n(1 + (\overline{y}+\alpha_0)j^B_i + a^{-1}J^B_i) \in \He_{v,v_0}(B_n) \otimes_\mathcal{A}\mathcal{R}.
\]
By Corollary \ref{sec:jucys-murphy-central}, the element $\beta_n(y)$ is central in $\He_{v,v_0}(B_n)\otimes_\mathcal{A}\mathcal{R}$. 
We have the following
\begin{theorem}
  \label{sec:gener-mark-trac-2}
  The equality $\tr_{n,y}^B = \tr_{\beta_n(y)}$ holds.
\end{theorem}
The property in Theorem \ref{sec:gener-mark-trac-1}~\ref{item:2} follows from the more general property (compare \cite[Proposition 4.5]{geckMarkovTracesKnot2004}).
\begin{proposition}
  \label{sec:gener-mark-trac-5}
 For any $h \in \He(B_{n-1})$, we have $\tr_{\beta_n(y)}(\iota(h)T_n) = y\tr_{\beta_{n-1}(y)}(h).$ 
\end{proposition}
\begin{proof}
   We will use the following lemma, verified by a direct computation.
  \begin{lemma}
    \label{sec:gener-mark-trac}
    \begin{enumerate}[label=(\alph*).,ref=(\alph*)]
    \item \label{item:8} We have $T_n^{-1} = T_n - \alpha_0.$
    \item \label{item:9} The elements $T_n, T_n^{-1}$ satisfy \[T_n \in \mathrm{span}_{\mathcal A}\langle t_w, w \in W(B_n)\backslash W(B_{n-1})\rangle\] and \[T_n^{-1} \in \mathrm{span}_{\mathcal A}\langle t_w^{-1}, w \in W(B_n)\backslash W(B_{n-1})\rangle.\]
    \item \label{item:10} The element $(j^B_n)^{-1}T_n$ satisfies \[(j^B_n)^{-1}T_n - 1\in \mathrm{span}_{\mathcal A}\langle t_w^{-1}, w \in W(B_n)\backslash W(B_{n-1})\rangle.\]
    \end{enumerate}
  \end{lemma}

  To simplify the formulas, introduce the notation $y' = y-\alpha_0, \overline{y'} = \overline{y}+\alpha_0$. We have
  \begin{equation*}
    \langle\beta_n(y), \iota(h)T_n\rangle = \langle\beta_{n-1}(y), \iota(h)T_n\rangle + {y'}\langle\beta_{n-1}(y)j^B_n,\iota(h)T_n\rangle +a\langle\beta_{n-1}(y)J_n^B, \iota(h)T_n\rangle.
  \end{equation*}
  Let us examine the summands on the right-hand side of this equation one by one. We will rutinely use the Corollary \ref{sec:trace-funct-pair-1}. For the first summand, we have, by Lemma \ref{sec:gener-mark-trac}~\ref{item:8}, \ref{item:9}, 
  \[
   \langle\beta_{n-1}(y), \iota(h)T_n\rangle = \langle\beta_{n-1}(y), \iota(h)T_{n}^{-1}\rangle + \alpha_0\langle\beta_{n-1}(y), \iota(h)\rangle = \alpha_0\langle\beta_{n-1}(y), \iota(h)\rangle.
  \]
  For the second summand, first note that $j^B_n$ is in the centralizer of $\He(X_{n-1})$ in $\He(X_n)$. It follows that
  \[
  {y'}\langle\beta_{n-1}(y)j^B_n,\iota(h)T_n\rangle = {y'}\langle\beta_{n-1}(y),\iota(h)(j_n^B)^{-1} T_n\rangle.
  \]
  Thus we have, by Lemma \ref{sec:gener-mark-trac}~\ref{item:10}, 
  \[
 {y'}\langle\beta_{n-1}(y),\iota(h)(j^{B}_n)^{-1}T_n\rangle = y'\langle\beta_{n-1}(y), \iota(h)\rangle.
  \]
  For the third summand, by Proposition \ref{sec:full-twist-jucys-1} and Lemma \ref{sec:gener-mark-trac}~\ref{item:9}, we have
  \[
    a\langle\beta_{n-1}(y)J_n^B, \iota(h)T_n\rangle = \overline{\langle\iota(h)T_n, \beta_{n-1}(y)\Se(B_{n-1})\rangle} = 0. 
  \]
  Collecting all three summands together, we get
  \[
    \langle\beta_n(y), \iota(h)T_n\rangle = (\alpha_0+y')\langle\beta_{n-1}(y),\iota(h)\rangle = y\langle\beta_{n-1}(y),\iota(h)\rangle,
  \]
  which completes the proof.
\end{proof}
\begin{proof}[Proof of Theorem \ref{sec:gener-mark-trac-2}.]
  It remains to check the properties \ref{item:5} and \ref{item:6}. Note that $j^B_n = t_{c}, j^B_nt_n^{-1} = t_{c'}$ for $c, c' \in W(B_n)\backslash W(B_{n-1})$. It follows from Corollary \ref{sec:trace-funct-pair-1}, that, for any $h\in \He(B_{n-1})$ we have
  \[
    \langle (\overline{y}+\alpha_0)\beta_{n-1}j_n^{B}, \iota(h) \rangle = 0 = \langle (\overline{y}+\alpha_0)\beta_{n-1}j_n^{B}, \iota(h)t_n\rangle,
  \]
  which reduces the proof to the proof of Theorem \ref{sec:markov-traces-via}.
\end{proof}
From Theorems \ref{sec:gener-mark-trac-1} and \ref{sec:gener-mark-trac-2} we get the following description of a Markov trace $\tr^D_{n,y}$.
\begin{corollary}
\label{sec:gener-mark-trac-4}  Let \[P_n(x_1, \dots, x_n;y) = \prod_{i=1}^n(1 + \overline{y}x_i + a^{-1}x_i^2),\] a symmetric polynomial in variables $(x_i)$ over $\mathcal{R}$. Write $P_n = P_n^+ + P_n^-$, a sum of even and odd monomial degree parts. Let $\delta_n(y) = P_n^+(j_1^D, \dots, j_n^D;y)$. Then $\tr_{n,y}^D = \tr_{\delta_n(y)}$.
\end{corollary}
\begin{proof}
  Recall that the trace $\tr_{n,(y_{\bullet})}^D$ is obtained by restriction of $\tr_{n,(y_{\bullet})}^B$ along the embedding $\iota_{D\to B}:\He(D_n) \to \He_{v,1}(B_n)$, so that we have
  \[
    \tr_{n,y}^D(h) = \langle \beta_n(y),\iota_{D\to B}(h)\rangle = \langle P(j_1^B,\dots,j_n^B;y),\iota_{D\to B}(h)\rangle.
  \]
  It remains to observe that for all $i$ we have $j_i^D = t_0j_i^B$, and so every even degree monomial in $j_i^B$ is equal to the same monomial in $j_i^D$, since $t_0^2 = 1$ in $\He_{v,1}(B_n)$ and $t_0$ commutes with $j_i^B$ for all $i$; and that every odd degree monomial in $j_i^B$ is orthogonal to $\iota_{D\to B}(h)$ with respect to $\langle\,,\rangle$. This completes the proof of the Corollary.
\end{proof}
We have the following analogue of Proposition \ref{sec:gener-mark-trac-5} in type $D$.
\begin{proposition}
  \label{sec:gener-mark-trac-6}
For $n \geq 1$ and any $h \in \He(D_{2n-2})$, we have \[\tr_{2n,y}^D(\iota\circ\iota(h)U_{2n-1}U_{2n}) = y^2\tr_{2n-2,y}^D(h).\]
\end{proposition}
\begin{proof}
 Note that we have $\iota_{D \to B}(U_k) = T_kt_0$, and so 
\[\tr_{2n,y}^D(\iota\circ\iota(h)U_{2n-1}U_{2n}) = \tr_{2n,y}^B(\iota\circ\iota(h)T_{2n-1}t_0T_{2n}t_0).\]
Using the property \ref{item:4} and Proposition \ref{sec:gener-mark-trac-5} twice, we get the result. 
\end{proof}
\subsection{Geometric Markov traces.}
\label{sec:geom-mark-trac}
Markov traces $\tr_n^B:\He_{v,v}(B_n) \to \mathcal{A}(a), \tr_n^D:\He(D_n) \to \mathcal{A}(a)$ studied in \cite{gomi2006markov}, \cite{websterGeometryMarkovTraces2009} are specializations of traces $\tr_{n,y}^X$ satisfying the extension of the properties \ref{item:5}, \ref{item:6} for the embeddings $B_0 \to B_1$ and $D_1 \to D_2, D_0 \to D_1$. In particular, these traces satisfy
\[
  \tr_1^B(t_0) = v-v^{-1}, \tr_1^D(t_1) = \tr_1^D(t_{1'}) = v-v^{-1}, \tr^D_2(t_1t_{1'}) = (v-v^{-1})^2.
\]
The condition on $\tr_1^B$ forces $y = v - v^{-1}$ and $\overline{y}-\alpha_0 = 0$. Thus we get that $\tr_n^B = \tr_{\zeta_n}$, the trace from Theorem \ref{sec:markov-traces-via}.

The conditions on $\tr_1^D, \tr_2^D$ imply that
\[
  y^2 = \tr_2^D(U_1U_2) = \tr_2^D(t_1t_{1'}^{-1}) = \tr_2^D(t_1t_{1'}) - \alpha \tr_2^D(t_1) = \alpha^2 - \alpha^2(1+a) = -\alpha^2a.
\]

It is easy to see that geometric Markov traces $\tr_n^B, \tr_n^D$ take values in $\mathcal{A}[a]$. Let $\tr_n^{X,k}:\He(X_n) \to \mathcal{A}$ be the coefficient of $\tr_n^X$ at $a^k$. We summarize the above discussion in the following
\begin{theorem}
  \label{sec:geom-mark-trac-1}
  The functional $\tr_n^{B,k}$ is represented by the $k$th elementary symmetric polynomial in Jucys-Murphy elements:
  \[
    \tr_n^{B,k}(h) = \langle e_k(J_1^B, \dots, J_n^B), h\rangle.
  \]

  Let $e_k'(x_1, \dots, x_n;\alpha)$ be the symmetric polynomial in $(x_i)$ appearing as a degree $2k$ homogeneous part in
  \[
    \prod_{i=1}^n(1 + \alpha x_i - x_i^2) = \prod_{i=1}^n(v-x_i)(v^{-1}+x_i).
  \]
  where $\alpha = v - v^{-1}$.
  Then
   \[
    \tr_n^{D,k}(h) = (-1)^k\langle e'_k(j_1^D, \dots, j_n^D; \alpha),h \rangle.
  \]
\end{theorem}
The symmetric polynomial $e_k'$ can be written explicitly as
\[
  e_k' = \sum_{\substack{i+j = 2k,\\ i,j\geq 0}}(-1)^{i}v^{j-i}e_ie_j.
\]
Note that the highest $a$-coefficient is always represented by the ``full twist'' element $\Se^{-1} = e_n(J_1^X, \dots, J_n^X), \tr_n^{X,n} = \tr_{\Se^{-1}}.$ This  recovers the symmetry of the HOMFLY-PT polynomial observed by K\'{a}lm\'{a}n \cite{kalmanMeridianTwistingClosed2009}, that was categorified in \cite{gorskySerreDualityKhovanov2019} in type $A$ and in \cite{bezrukavnikovMonodromicModelKhovanov2022} in general type.
\begin{remark}
  The statement analogous to Theorem \ref{sec:geom-mark-trac-1} was lifted to the level of Khavanov-Rozansky homology in \cite{bezrukavnikovMonodromicModelKhovanov2022}. Namely, there it is shown that the object representing the $k$th degree of the invariant has a filtration by the objects representing products of Jucys-Murphy elements. For an arbitrary Weyl group, the object representing the $k$th degree was constructed in \cite{bezrukavnikovMonodromicModelKhovanov2022} in terms of some explicit character sheaves on a reductive group $G$ with the Weyl group $W$. It was also explained how these character sheaves correspond to exterior powers of the reflection representation of $W$. We expect that, if the analogous filtration categorifying expressions in Theorem \ref{sec:geom-mark-trac-1} can be constructed in types $B$ and $D$, the arguments in this paper can be lifted to the level of Khovanov-Rozansky homology: the pairing $\langle, \rangle$ is categorified to the $\mathrm{Hom}$-pairing on the Hecke category; the ``Serre property'' of Proposition \ref{sec:full-twist-jucys-1} has a categorical analogue in the description of the Serre functor of the Hecke category \cite{beilinsonTiltingExercises2004}, \cite{gorskySerreDualityKhovanov2019}; elements $j_n^X$ can be expressed using the braid group action on the Hecke category. Such an extension would, for example, give a geometric interpretation to the special properties of Markov traces in types $B_n$ and $D_n$ involving the elements $T_n, U_n$. We hope to address this in a future work.

  We would also like to  remark that, in type $A$, the fact that Jucys-Murphy elements and their symmetric polynomials come from the extended affine Hecke algebra has an interpretation in terms of geometric representation theory, see \cite{gnr}, \cite{oblomkovAFFINEBRAIDGROUP2019} and \cite{tolmachovkostiantynLinearStructureFinite}. We do not know if this interpretation can be extended to other types. 
\end{remark}

\bibliography{bibliography}
\bibliographystyle{alpha}

\Addresses

\end{document}